\documentclass[reqno]{amsart}
\usepackage[american]{babel}
\usepackage{amsmath,amsfonts,amsthm,amssymb,latexsym,euscript}
\numberwithin{equation}{section}

\newtheorem{theorem}{Theorem}[section]
\newtheorem{lemma}{Lemma}[section]
\newtheorem{corollary}{Corollary}[section]
\theoremstyle{definition}
\newtheorem{definition}{Definition}[section]
\theoremstyle{remark}
\newtheorem{remark}{Remark}[section]

\DeclareMathOperator{\sgn}{sign}

\begin{document}

\title[Zakharov--Kuznetsov equation]
{An initial-boundary value problem for three-dimensional Zakharov--Kuznetsov equation}

\author[A.V.~Faminskii]{Andrei~V.~Faminskii}


\address{Department of Mathematics, Peoples' Friendship University of Russia,
Miklukho--Maklai str. 6, Moscow, 117198, Russia}
\email{afaminskii@sci.pfu.edu.ru}

\date{}
\maketitle

{\scriptsize \centerline{Peoples' Friendship University of Russia, Moscow, Russia}}

\begin{abstract}
An initial-boundary value problem with homogeneous Dirichlet boundary conditions for three-dimensional Zakharov--Kuznetsov equation is considered. Results on global existence, uniqueness and large-time decay of weak solutions in certain weighted spaces are established.
\end{abstract}

\section{Introduction. Description of main results}\label{S1}

Three-dimensional Zakharov--Kuznetsov equation (ZK) 
\begin{equation}\label{1.1}
u_t+bu_x+ u_{xxx}+u_{xyy}+u_{xzz}+uu_x=f(t,x,y,z)
\end{equation}
($u=u(t,x,y,z)$, $b$ -- real constant) for the first time was derived in \cite{ZK} for description of ion-acoustic waves in plasma put in the magnetic field. Further, this equation became to be considered as a model equation for non-linear waves propagating in dispersive media in the preassigned direction $(x)$ with deformations in the transverse directions. A rigorous derivation of the ZK model can be found, for example, in \cite{LLS}. Zakharov--Kuznetsov equation generalizes Korteweg--de~Vries equation (KdV) $u_t+bu_x+u_{xxx}+uu_x=0$ in the multidimensional case.

In the present paper we consider an initial-boundary value problem on a layer $\Sigma =\mathbb R\times \Omega$, where $\Omega$ is a certain bounded domain in $\mathbb R^2$, with initial and boundary conditions
\begin{equation}\label{1.2}
u\big|_{t=0} =u_0(x,y,z),
\end{equation}
\begin{equation}\label{1.3}
u\big|_{(0,T)\times \partial\Sigma} =0, 
\end{equation}
where $T>0$ is arbitrary. We establish results on global existence, uniqueness and large-time decay of weak solutions to this problem. Existence and uniqueness of weak solutions are also obtained for the initial value problem.

The theory of ZK equation is more or less developed in the two-dimensional case, that is for an equation
$$
u_t+bu_x+ u_{xxx}+u_{xyy}+uu_x=f(t,x,y),
$$
especially for the initial-value problem. In particular, classes of global well-posedness were constructed in \cite{F95} for  initial data from the spaces $H^k(\mathbb R^2)$, $k\in\mathbb N$. Other results can be found in \cite{S1, F89, LP, LPS}.
Initial-boundary value problems on domains of the type $I\times \mathbb R$, where $I$ is a certain interval (bounded or unbounded), are studied in \cite{F02, F07, FB, F08, ST, F12, DL} and others. Initial-boundary value problems for $y$ varying in a bounded interval turned out to be the most complicated ones (\cite{LPS, LT, L13, BF, STW, DL}), although such problems seems to be more natural from the physical point of view. In particular, there are no results on global well-posedness in classes of regular solutions for the strips $\mathbb R\times I$.

The theory of equation \eqref{1.1} is on the initial level. Certain results on global existence of weak solutions (without uniqueness) for the initial value problem follow from \cite{S1, F89}. Local well-posedness for initial data from 
$H^s(\mathbb R^3)$, $s>1$, is established in \cite{LS, RS}. Results similar to \cite{S1, F89}, that is global existence without uniqueness of weak solutions for initial-boundary value problems on domains of the type $I\times \mathbb R^2$ can be found in \cite{F02, FB, ST, F12} and on a bounded rectangle in \cite{STW, W2}. Regular solutions to one initial-boundary value problem on a bounded rectangle are considered in \cite{W1, L14} and in the last paper global regular solutions are constructed for small initial data.

Homogeneous equation \eqref{1.1} possesses two conservation laws for solutions to the considered problem:
\begin{equation}\label{1.4}
\iint_{\Sigma}u^2\,dxdydz=\text{const}, \mbox{  }\iint_{\Sigma}\left(u_x^2+u^2_y+u^2_z-
\frac 13 u^3\right)\,dxdydz=\text{const}. 
\end{equation}
Of course, similar conservation laws exist for the initial value problem for homogeneous KdV equation. It is well-known that the number of conservation laws for KdV is infinite, while other ones for ZK are not found. The last circumstance, for example, did not allow to apply in \cite{LS, RS} their profound investigations of the linearized equation to establish global well-posedness. In the present paper we supplement these two conservation laws with some decay of solutions when $x\to +\infty$ and construct classes of global existence and uniqueness without any assumptions on the size of the initial data. According to our best knowledge it is the first result of such a type for equation \eqref{1.1}. For KdV such method was for the first time used in \cite{KF, K, F88}. In the two-dimensional case similar results for ZK were obtained in \cite{BF}.

In all the consequent results the domain $\Omega$ is bounded and satisfy the following assumption:

\noindent either 1) $\partial\Omega \in C^3$ (in the conventional sense, see, for example, \cite{Mikh}),

\noindent or 2) $\Omega$ is a rectangle $(0,L_1)\times (0,L_2)$ for certain positive $L_1, L_2$   

\noindent (one can introduce certain more complicated assumptions on $\Omega$ such that the aforementioned domains are particular cases of more general ones and all the results of the paper hold but for simplicity we choose this variant). 
The symbol $|\Omega|$ denotes the measure of $\Omega$.

Let $\Pi_T =(0,T)\times \Sigma$, $x_+=\max(x,0)$, $\mathbb R_+=(0,+\infty)$, $\Sigma_+=\mathbb R_+ \times \Omega$. 

For an integer $k\geq 0$ let
$$
|D^k\varphi|=\Bigl(\sum_{k_1+k_2+k_3=k}(\partial^{k_1}_x\partial_y^{k_2}\partial_z^{k_3}\varphi)^2\Bigr)^{1/2}, \qquad |D\varphi|=|D^1\varphi|.
$$
 
Let $L_p=L_p(\Sigma)$, $L_{p,+}=L_p(\Sigma_+)$, $H^k=H^k(\Sigma)$, $H_0^1=H_0^1(\Sigma)=\{\varphi\in H^1: 
\varphi|_{\partial\Sigma}=0\}$ (note that under the aforementioned assumptions on $\Omega$ the space $H_0^1$ coincides with the closure of the space $C_0^\infty(\Sigma)$ in the $H^1$-norm).

For a measurable non-negative on $\mathbb R$ function $\psi(x)\not\equiv \text{const}$, let
\begin{equation*}
L_2^{\psi(x)} =\{\varphi(x,y,z): \varphi\psi^{1/2}(x)\in L_2\}
\end{equation*}
with a natural norm. In particularly important cases we use the special notation
\begin{equation*}
L_2^\alpha=L_2^{(1+x_+)^{2\alpha}}\quad \forall\ \alpha\ne 0,\quad 
L_2^0=L_2,\qquad
L_2^{\alpha,exp}=L_2^{1+e^{2\alpha x}}\quad \forall\  \alpha>0.
\end{equation*}
Restrictions of these spaces on $\Sigma_+$ are denoted by $L_{2,+}^{\psi(x)}$, $L_{2,+}^\alpha$, 
$L_{2,+}^{\alpha,exp}$.

Let for an integer $k\geq 0$
\begin{equation*}
H^{k,\psi(x)}=\{\varphi: |D^j\varphi|\in L_2^{\psi(x)}, \  j=0,\dots,k\}
\end{equation*}
with a natural norm,
\begin{equation*}
H^{k,\alpha}=H^{k,(1+x_+)^{2\alpha}}\quad \forall\ \alpha\ne 0,\quad H^{k,0}=H^k,
\qquad H^{k,\alpha,exp}=H^{k,1+e^{2\alpha x}} \quad \forall\  \alpha>0.
\end{equation*}
Let $H^{1,\psi(x)}_0=\{\varphi\in H^{1,\psi(x)}: \varphi|_{\partial\Sigma}=0\}$ with similar notation for 
$H_0^{1,\alpha}$ and $H_0^{1,\alpha,exp}$. Let $H^{-1,\psi(x)}= \{\varphi: \varphi\psi^{1/2}(x)\in H^{-1}\}$.

We say that $\psi(x)$ is an admissible weight function if $\psi$ is an infinitely smooth positive function on $\mathbb R$ such that $|\psi^{(j)}(x)|\leq c(j)\psi(x)$ for each natural $j$ and all $x\in\mathbb R$. Note that such a function has not more than exponential growth and not more than exponential decrease at $\pm\infty$. It was shown in \cite{F12} that $\psi^s(x)$ for any $s\in\mathbb R$ is also an admissible weight function.

As an important example of such functions, we introduce for $\alpha\geq 0$ special infinitely smooth functions $\rho_\alpha(x)$ in the following way: $\rho_\alpha(x)=1+e^{2x}$ when $x\leq -1$,  $\rho_\alpha(x)=1+(1+x)^{2\alpha}$ for $\alpha>0$ and $\rho_0(x)=3-(1+x)^{-1/2}$ when $x\geq0$, $\rho'_\alpha(x)>0$ when $x\in (-1,0)$. 

Note that both $\rho_\alpha$ and $\rho'_\alpha$ are admissible weight functions and $\rho'_\alpha(x) \leq c(\alpha)\rho_\alpha(x)$ for all $x\in\mathbb R$.
 Moreover, for $\alpha\geq 0$
\begin{equation*}
L_2^{\rho_\alpha(x)}=L_2^\alpha, \qquad H^{k,\rho_\alpha(x)}=H^{k,\alpha}.
\end{equation*}

We construct solutions to the considered problem in spaces $X^{k,\psi(x)}(\Pi_T)$,\linebreak $k=0 \mbox{ or }1$, for admissible non-decreasing weight functions $\psi(x)\geq 1\ \forall x\in\mathbb R$, consisting of functions $u(t,x,y,z)$ such that
\begin{equation}\label{1.5}
u\in C_w([0,T]; H^{k,\psi(x)}), \qquad
|D^{k+1}u|\in L_2(0,T;L_2^{\psi'(x)})
\end{equation}
(the symbol $C_w$ denotes the space of  weakly continuous mappings),
\begin{equation}\label{1.6}
\lambda(|D^{k+1} u|;T) =
\sup_{x_0\in\mathbb R}\int_0^T\!\! \int_{x_0}^{x_0+1}\!\! \iint_\Omega |D^{k+1}u|^2\,dydzdxdt<\infty, \quad
u\big|_{(0,T)\times\partial\Sigma}=0
\end{equation}
(let $X^{\psi(x)}(\Pi_T)=X^{0,\psi(x)}(\Pi_T)$). 

In particularly important cases we use the special notation
\begin{equation*}
X^{k,\alpha}(\Pi_T)=X^{k,\rho_\alpha(x)}(\Pi_T),\quad X^\alpha(\Pi_T)=X^{0,\alpha}(\Pi_T)
\end{equation*}
and for $\alpha>0$
\begin{equation*}
X^{k,\alpha,exp}(\Pi_T)=X^{k,1+e^{2\alpha x}}(\Pi_T),\quad X^{\alpha,\exp}(\Pi_T)=X^{0,\alpha,\exp}(\Pi_T).
\end{equation*}
It is easy to see that $X^{k,0}(\Pi_T)$ coincides with a space of functions $u\in C_w([0,T]; H^k)$ for which \eqref{1.6} holds, $X^{k,\alpha}(\Pi_T)$, $\alpha>0$, -- with a space of functions $u\in C_w([0,T]; H^{k,\alpha})$  for which \eqref{1.6} holds and, in addition, $|D^{k+1} u|\in L_2(0,T;L_{2,+}^{\alpha-1/2})$; $X^{k,\alpha,exp}(\Pi_T)$ -- with a space of functions $u\in C_w([0,T]; H^{k,\alpha,exp})$  for which \eqref{1.6} holds and, in addition,
$|D^{k+1} u|\in L_2(0,T;L_{2,+}^{\alpha,exp})$.

\begin{theorem}\label{T1.1}
Let $u_0\in L_2^{\psi(x)}$, $f\in L_1(0,T; L_2^{\psi(x)})$ for  certain $T>0$ and an admissible weight function $\psi(x)\geq 1\ \forall x\in\mathbb R$ such that $\psi'(x)$ is also an admissible weight function. Then there exists a weak solution to problem \eqref{1.1}--\eqref{1.3} $u \in X^{\psi(x)}(\Pi_T)$. 
\end{theorem}

\begin{theorem}\label{T1.2}
Let $u_0\in H_0^{1,\psi(x)}$, $f\in L_1(0,T; H_0^{1,\psi(x)})$ for  certain $T>0$ and an admissible weight function $\psi(x)\geq 1\ \forall x\in\mathbb R$ such that $\psi'(x)$ is also an admissible weight function. Then there exists a weak solution to problem \eqref{1.1}--\eqref{1.3} $u\in X^{1,\psi(x)}(\Pi_T)$ and it is unique in this space if $\psi(x)\geq \rho_{3/4}(x)$ $\forall x\in \mathbb R$.  
\end{theorem}

\begin{remark}\label{R1.1}
It follows from Theorem~\ref{T1.2} that weak solutions to problem \eqref{1.1}--\eqref{1.3} are unique in the spaces
$X^{1,3/4}(\Pi_T)$ and in the spaces $X^{1,\alpha,exp}(\Pi_T)$ for any $\alpha>0$ (and exist under corresponding assumptions on $u_0$ and $f$).
\end{remark}

For small solutions to the considered problem the following large-time decay result holds.

\begin{theorem}\label{T1.3}
Let $\Omega_0=+\infty$ if $b\leq 0$, and if $b>0$ there exists $\Omega_0>0$ such that in both cases if 
$|\Omega| <\Omega_0$ there exist $\alpha_0>0$, $\epsilon_0>0$ and $\beta>0$ such that if $u_0\in L_2^{\alpha,exp}$ for $\alpha\in (0,\alpha_0]$, $\|u_0\|_{L_2}\leq\epsilon_0$,  $f\equiv 0$, then there exists a weak solution $u(t,x,y,z)$ to problem \eqref{1.1}--\eqref{1.3} from the space $X^{\alpha,exp}(\Pi_T)$ $\forall T>0$ satisfying an inequality
\begin{equation}\label{1.7}
\|e^{\alpha x}u(t,\cdot,\cdot,\cdot)\|_{L_2}\leq e^{-\alpha\beta t}\|e^{\alpha x}u_0\|_{L_2}\qquad \forall t\geq 0.
\end{equation}
\end{theorem}

The proof of this result, in particular, is based on Friedrichs inequality and, therefore, homogeneous Dirichlet conditions are essential. The idea that under such conditions Zakahrov--Kuznetsov equation possesses certain internal dissipation, which provides decay of such a type, was found out in \cite{LT}. Stabilization of solutions to three-dimensional linearized ZK equation is studied in \cite{DL15}.

Further we use the following auxiliary functions. Let $\eta(x)$ denote a cut-off function, namely, $\eta$ is an infinitely smooth non-decreasing on $\mathbb R$ function such that $\eta(x)=0$ when $x\leq 0$, $\eta(x)=1$ when $x\geq 1$, $\eta(x)+\eta(1-x)\equiv 1$.

For each $\alpha\geq 0$ and $\beta>0$ we introduce an infinitely smooth increasing on $\mathbb R$ function $\varkappa_{\alpha,\beta}(x)$ as follows: 
$\varkappa_{\alpha,\beta}(x)=e^{2\beta x}$ when $x\leq -1$, $\varkappa_{\alpha,\beta}(x)=(1+x)^{2\alpha}$ for $\alpha>0$ and $\varkappa_{0,\beta}(x)=2-(1+x)^{-1/2}$ when $x\geq 0$, $\varkappa'_{\alpha,\beta}(x)>0$ when $x\in (-1,0)$. 

Note that both $\varkappa_{\alpha,\beta}$ and $\varkappa'_{\alpha,\beta}$ are admissible weight functions, and $\varkappa'_{\alpha,\beta}(x)\leq c(\alpha,\beta)\varkappa_{\alpha,\beta}(x)$ for all $x\in \mathbb R$. It is obvious that one can take $\rho_\alpha(x) \equiv 1+\varkappa_{\alpha,1}(x)$.

Note also that if $u\in X^{k,\alpha}(\Pi_T)$ for $\alpha\geq 1/2$, then $|D^{k+1}u|\varkappa_{\alpha-1/2,\beta}^{1/2}(x)\in L_2(\Pi_T)$ for any $\beta>0$.

Further we omit limits of integration in integrals over the whole strip $\Sigma$. We need the following interpolating inequality.

\begin{lemma}\label{L1.1}
Let $\psi_1(x)$, $\psi_2(x)$ be two admissible weight functions such that $\psi_1(x)\leq c_0\psi_2(x)$ 
$\forall x\in\mathbb R$ for some constant $c_0>0$. Let $k=1$ or $2$, $m\in [0,k)$ -- integer, $q\in [2,6]$ if $k-m=1$ and  $q\in [2,+\infty)$ if $k-m=2$. Then there exists a constant $c>0$ such that for every function $\varphi(x,y,z)$ satisfying  
$|D^k\varphi|\psi_1^{1/2}(x)\in L_2$, $\varphi\psi_2^{1/2}(x)\in L_2$, the following inequality holds
\begin{equation}\label{1.8}
\bigl\| |D^m\varphi|\psi_1^s(x)\psi_2^{1/2-s}(x)\bigr\|_{L_q} \leq c 
\bigl\| |D^k\varphi|\psi_1^{1/2}(x)\bigr\|^{2s}_{L_2}
\bigl\| \varphi\psi_2^{1/2}(x)\bigr\|^{1-2s}_{L_2} 
+\bigl\| \varphi\psi_2^{1/2}(x)\bigr\|_{L_2},
\end{equation}
where $\displaystyle{s=s(k,m,q)=\frac{2m+3}{4k}-\frac{3}{2kq}}$. If $\varphi\big|_{\partial\Sigma}=0$ and either $k=1$ or $k=2,$ $m=0,$ $q\leq 6$ or $k=2,$ $ m=1,$ $q=2$ then the constant $c$ in \eqref{1.8} does not depend on $\Omega$.
\end{lemma}

\begin{proof}
Let first $k=1$. The proof is based on the well-known inequality (of course, it is valid for more general domains): for 
$\varphi\in W^1_p$, $p\in [1,3)$, $p^*=3p/(3-p)$
\begin{equation}\label{1.9}
\|\varphi\|_{L_{p^*}} \leq c \bigl\||D\varphi|+|\varphi|\bigr\|_{L_p},
\end{equation}
where the constant $c$ does not depend on $\Omega$ in the case $\varphi\in H^1_0$ (see, for example, \cite{BIN,LSU}).
Then for $q\in [2,6]$ H\"older inequality yields that
$$
\|\varphi \psi_1^s \psi_2^{1/2-s}\|_{L_q} \leq \|\varphi\psi_1^{1/2}\|_{L_6}^{2s} 
\|\varphi \psi_2^{1/2}\|_{L_2}^{1-2s},
$$
whence with the use of \eqref{1.9} for $p=2$ and the properties of the functions $\psi_1$ and $\psi_2$ the desired estimate follows.

Next, let $k=2,$ $m=1,$ $q=2$. Integration by parts yields an equality
\begin{multline*}
\iiint |D\varphi|^2\psi_1^{1/2}\psi_2^{1/2}\,dxdydz = 
-\iiint \Delta\varphi \psi_1^{1/2} \cdot \varphi\psi_2^{1/2}\,dxdydz  \\
-\iiint \varphi\varphi_x (\psi_1^{1/2}\psi_2^{1/2})'\,dxdydz +
\iint_{\partial\Sigma} \varphi (\varphi_y n_y +\varphi_z n_z)\psi_1^{1/2}\psi_2^{1/2} dS,
\end{multline*}
where $(n_y,n_z)$ is the exterior normal vector to $\Omega$. If $\varphi\big|_{\partial\Sigma}=0$ this equality immediately provides \eqref{1.8}, while in the general case one must also use for functions $\Phi\equiv \varphi^2$ and $\Phi\equiv\varphi_y^2$ or $\Phi\equiv \varphi_z^2$ the following well-known estimate on the trace (see, for example, \cite{BIN}):
$$
\|\Phi\|_{L_1(\partial\Omega)} \leq c \bigl\| |\Phi_y|+|\Phi_z|+|\Phi| \bigr\|_{L_1(\Omega)}.
$$

If $k=2,$ $m=1,$ $q\in (2,6]$ let $\sigma= \displaystyle \frac 32 -\frac 3q$, then
$$
\bigl\| |D\varphi|\psi_1^s\psi_2^{1/2-s}\bigr\|_{L_q} \leq
\bigl\| |D\varphi|\psi_1^{1/2}\bigr\|_{L_6}^\sigma 
\bigl\||D\varphi|\psi_1^{1/4}\psi_2^{1/4} \bigr\|_{L_2}^{1-\sigma}
$$
and with the use of the already obtained estimates \eqref{1.8} for $k=1$ applied to $|D\varphi|$ and for $k=2,$ $m=1,$ $q=2$ derive \eqref{1.8} in this case.

Finally, let $k=2,$ $m=0$. If $q\leq 6$ then with the use of \eqref{1.8} for $k=1$ (where $\psi_1$ is substituted by
$\psi_1^{1/2}\psi_2^{1/2}$) and for $k=2,$ $m=1,$ $q=2$ we derive that
\begin{multline*}
\|\varphi \psi_1^s\psi_2^{1/2-s}\|_{L_q} =
\| \varphi(\psi_1^{1/2}\psi_2^{1/2})^{s(1,0,q)} \psi_2^{1/2-s(1.0,q)}\|_{L_q}  \\  \leq
c\bigl\| |D\varphi|\psi_1^{1/4}\psi_2^{1/4}\bigr\|_{L_2}^{2s(1,0,q)}
\|\varphi\psi_2^{1/2}\|_{L_2}^{1-2s(1,0,q)} +c\|\varphi\psi_2^{1/2}\|_{L_2} \\ \leq
c_1\bigl\| |D^2\varphi|\psi_1^{1/2}\bigr\|_{L_2}^{2s}
\|\varphi\psi_2^{1/2}\|_{L_2}^{1-2s} +c_1\|\varphi\psi_2^{1/2}\|_{L_2}.
\end{multline*}
If $q\in (6,+\infty)$ choose $p\in (2,3)$ satisfying $q=p^*$, $\varkappa=1+(6-p)/(3p)$ and $\theta\in (5/6,1)$ satisfying
$\displaystyle \frac 1q = \frac \theta{\varkappa q} +\frac{1-\theta}2$ (of course, all these parameters can be expressed explicitly) and define
$$
\widetilde\varphi \equiv |\varphi|^\varkappa \cdot \sgn\varphi\cdot \psi_1^{1/2}\psi_2^{(\varkappa-1)/2}.
$$
It is easy to see that $s=\displaystyle\frac \theta{2\varkappa}$ and thus
$$
\|\varphi\psi_1^s\psi_2^{1/2-s}\|_{L_q} =
\|\widetilde\varphi^{\theta/\varkappa}\cdot (\varphi\psi_2^{1/2})^{1-\theta}\|_{L_q}
\leq \|\widetilde\varphi\|_{L_q}^{\theta/\varkappa}
\|\varphi\psi_2^{1/2}\|_{L_2}^{1-\theta}.
$$
Applying inequality \eqref{1.9} to the function $\widetilde\varphi$ we find that
\begin{multline*}
\|\widetilde\varphi\|_{L_q} \leq c\bigl\||D\widetilde\varphi|+|\widetilde\varphi|\bigr\|_{L_p} \leq
c_1\bigl\|(|D\varphi|+|\varphi|)\psi_1^{1/2}\cdot (|\varphi|\psi_2^{1/2})^{\varkappa-1}\bigr\|_{L_p} \\ \leq
c_2\bigl\|(|D\varphi|+|\varphi|)\psi_1^{1/2}\bigr\|_{L_6} 
\|\varphi\psi_2^{1/2}\|_{L_2}^{\varkappa-1}.
\end{multline*}
Applying inequality \eqref{1.8} in the case $k=1$ to the function $|D\varphi|+|\varphi|$ we finish the proof.
\end{proof}

\begin{remark}\label{R1.2}
In the case $\psi_1=\psi_2\equiv 1$ inequality \eqref{1.8} is well-known (see, for example, \cite{BIN, LSU}. For the weighted spaces in the case $\Sigma=\mathbb R^3$ it was proved in \cite{F89} (in fact, in that paper the spatial dimension and natural $k$ were arbitrary).
\end{remark}

\begin{remark}\label{R1.3}
The constant $c$ in the right side of \eqref{1.8} depends on the corresponding constants evaluating the derivatives of the functions $\psi_1$ and $\psi_2$ by these functions themselves and the constant $c_0$ evaluating $\psi_1$ by $\psi_2$.
\end{remark}

For certain multi-index $\nu=(\nu_1,\nu_2)$ let $\partial^\nu_{y,z} =
\partial^{\nu_1}_y\partial^{\nu_2}_z$, $|\nu|=\nu_1+\nu_2$. Let $\Delta^{\bot} =\partial^2_y+\partial^2_z$, 
$\Delta = \partial^2_x +\Delta^{\bot}$.

The paper is organized as follows. An auxiliary linear problem is considered in Section~\ref{S2}. Section~\ref{S3} is dedicated to problems on existence of solutions to the original problem. Results on continuous dependence of solutions on $u_0$ and $f$ are proved in Section~\ref{S4}. In particular, they imply uniqueness of the solution. Section~\ref{S5} is devoted to the large-time decay of solutions. The initial value problem is considered in Section~\ref{S6}.

\section{An auxiliary linear equation}\label{S2}

Consider a linear equation 
\begin{equation}\label{2.1}
u_t+bu_x+\Delta u_x-\delta \Delta u=f(t,x,y,z)
\end{equation}
for a certain constant $\delta\in [0,1]$.

\begin{lemma}\label{L2.1}
Let $(1+|x|)^n \partial^j_x\partial^\nu_{y,z} u_0\in L_2$ for any integer non-negative $n$, $j$ and $|\nu|\leq 3$, $u_0\big|_{\partial\Sigma}=\Delta^{\bot}u_0\big|_{\partial\Sigma}=0$, 
$(1+|x|)^n \partial^m_t \partial^j_x\partial^\nu_{y,z}f\in L_1(0,T; L_2)$ for any integer $n$, $j$ and 
$2m+|\nu|\leq 3$, $f\big|_{(0,T)\times\partial\Sigma}=\Delta^{\bot} f\big|_{(0,T)\times\partial\Sigma}=0$.
Then there exists a solution to problem \eqref{2.1}, \eqref{1.2}, \eqref{1.3} $u(t,x,y,z)$ such that $(1+|x|)^n \partial^m_t \partial^j_x\partial^\nu_{y,z}u\in C([0,T]; L_2)$ for any integer $n$, $j$ and 
$2m+|\nu|\leq 3$, $\Delta^{\bot} u\big|_{(0,T)\times\partial\Sigma}=0$.
\end{lemma}

\begin{proof}
Let $\{\psi_l(y,z)$, $l=1,2\dots\}$ be an orthonormal in $L_2(\Omega)$ system of eigenfunctions for the operator $-\Delta^{\bot}$ on $\Omega$ with boundary conditions   $\psi_l\big|_{\partial\Omega}=0$, $\lambda_l$ -- the corresponding eigenvalues. It is known (see, for example \cite{Mikh}) that such a system exists and satisfy the following properties: $\lambda_l>0$ $\forall l$, $\lambda_l\to +\infty$ when $l\to +\infty$, $\psi_l\in H^3(\Omega)$,
$\psi_l\big|_{\partial\Omega}=\Delta^{\bot}\psi_l\big|_{\partial\Omega}=0$ $\forall l$ and these functions are real-valued. If $(\varphi,\psi_l)$ denotes the scalar product in $L_2(\Omega)$ then for any 
$\varphi\in L_2(\Omega)$
\begin{equation}\label{2.2}
\varphi =\sum\limits_{l=1}^{+\infty} (\varphi,\psi_l)\psi_l.
\end{equation}
If $\varphi \in H_0^1(\Omega)$ or $\varphi \in H^2(\Omega)\cap H_0^1(\Omega)$ then this series converges in these spaces respectively. Moreover,
\begin{gather}\label{2.3}
\varphi \in H_0^1(\Omega) \Longleftrightarrow \sum\limits_{l=1}^{+\infty} \lambda_l (\varphi,\psi_l)^2 <+\infty,
\quad \sum\limits_{l=1}^{+\infty} \lambda_l (\varphi,\psi_l)^2 =\|\varphi\|^2_{H_0^1(\Omega)}, \\
\label{2.4}
\varphi \in H^2(\Omega)\cap H_0^1(\Omega) \Longleftrightarrow 
\sum\limits_{l=1}^{+\infty} \lambda_l^2 (\varphi,\psi_l)^2 <+\infty, \quad
\sum\limits_{l=1}^{+\infty} \lambda_l^2 (\varphi,\psi_l)^2 \sim \|\varphi\|^2_{H^2(\Omega)}.
\end{gather}
The last inequality is the particular case of an inequality valid for any function $\varphi \in H^2(\Omega)\cap H_0^1(\Omega)$:
\begin{equation}\label{2.5}
\|\varphi\|_{H^2(\Omega)} \leq c(\Omega)\|\Delta^{\bot}\varphi\|_{L_2(\Omega)},
\end{equation}
where the constant $c$ depends on the domain $\Omega$. Moreover, for any function $\varphi\in H^3(\Omega)$, such that $\varphi\big|_{\partial\Omega}=\Delta^{\bot} \varphi\bigl|_{\partial\Omega}=0$,
series \eqref{2.2} converges in this space and similarly to \eqref{2.3}, \eqref{2.4} 
\begin{multline}\label{2.6}
\varphi\in H^3(\Omega), \varphi\big|_{\partial\Omega}=\Delta^{\bot} \varphi\bigl|_{\partial\Omega}=0
\Longleftrightarrow \sum\limits_{l=1}^{+\infty} \lambda_l^3 (\varphi,\psi_l)^2 <+\infty, \\
\sum\limits_{l=1}^{+\infty} \lambda_l^3 (\varphi,\psi_l)^2 \sim \|\varphi\|^2_{H^3(\Omega)}.
\end{multline}
For example, in the case $\Omega=(0,L_1)\times (0,L_2)$ these eigenfunctions are written in a simple form
 $\displaystyle \left\{\frac 2{\sqrt{l_1l_2}} \sin\frac{\pi l_1 y}{L_1} \sin\frac{\pi l_2 z}{L_2}, l_1,l_2=1,2,\dots\right\}$.

Then with the use of Fourier transform for the variable $x$ and Fourier series for the variables $y, z$  a solution to  problem \eqref{2.1}, \eqref{1.2}, \eqref{1.3} can be written as follows:
\begin{equation}\label{2.7}
u(t,x,y,z)=\frac{1}{2\pi}\int\limits_{\mathbb R}\sum_{l=1}^{+\infty}e^{i\xi x}\psi_l(y,z)\widehat{u}(t,\xi,l)\,d\xi,
\end{equation}
where
\begin{multline*}
\widehat{u}(t,\xi,l)=\widehat{u_0}(\xi,l)e^{\bigl(i(\xi^3-b\xi+\xi\lambda_l)
-\delta(\xi^2+\lambda_l)\bigr)t} \\+
\int_0^t\widehat{f}(\tau,\xi,l)e^{\bigl(i(\xi^3-b\xi+\xi\lambda_l)-\delta(\xi^2+\lambda_l)\bigr)(t-\tau)}\,d\tau,
\end{multline*}
$$
\widehat{u_0}(\xi,l)\equiv\iiint e^{-i\xi x}\psi_l(y,z)u_0(x,y,z)\,dxdydz,
$$
$$
\widehat f(t,\xi,l)\equiv\iiint e^{-i\xi x}\psi_l(y,z) f(t,x,y,z)\,dxdydz.
$$
According to \eqref{2.3}--\eqref{2.6} and the properties of the functions $u_0$ and $f$ the function $u$ is the desired solution.
\end{proof}

\begin{lemma}\label{L2.2}
Let the hypothesis of Lemma~\ref{L2.1} be satisfied and, in addition, $\partial_x^j \partial_{y,z}^\nu u_0e^{\alpha x} \in L_{2,+}$, $\partial_x^j \partial_{y,z}^\nu fe^{\alpha x}\in L_2(0,T; L_{2,+})$ for any $\alpha>0$, 
$j\geq 0$ and $|\nu|\leq 1$. Then $\partial_x^j \partial_{y,z}^\nu u e^{\alpha x} \in C([0,T]; L_{2,+})$ if 
$|\nu|\leq1$, $\partial^m_t\partial_x^j \partial_{y,z}^\nu u e^{\alpha x} \in L_2(0,T;L_{2,+})$ if $2m+|\nu|= 2$ also for any $\alpha>0$ and $j\geq 0$, where $u$ is the solution to problem \eqref{2.1}, \eqref{1.2}, \eqref{1.3} constructed in Lemma~\ref{L2.1}.
\end{lemma}

\begin{proof}
Let $v\equiv \partial^j_x u$, then the function $v$ satisfies an equation of \eqref{2.1} type, where $f$ is replaced by $\partial^j_x f$. Let $m\geq 3$. Multiplying this equation by $2x^m v$ and integrating over $\Sigma_+$, we derive an equality 
\begin{multline}\label{2.8}
\frac{d}{dt}\iiint_{\Sigma_+}x^mv^2\, dxdydz +
m\iiint x^{m-1} (3v_x^2+v_y^2+v_z^2-bv^2)\,dxdydz \\ -
m(m-1)(m-2)\iiint_{\Sigma_+}x^{m-3}v^2 \,dxdydz +
2\delta \iiint_{\Sigma_+} x^m (v_x^2+v_y^2+v_z^2)\,dxdydz \\ -
\delta m(m-1)\iiint_{\Sigma_+}x^{m-2}v^2\,dxdyz =2\iiint_{\Sigma_+}x^m\partial^j_xf v\,dxdydz.
\end{multline}
Let $\alpha>0$, $n\geq3$. For any $m\in [3,n]$ multiplying the corresponding inequality by $\alpha^m/(m!)$ and summing by $m$ we obtain that for
\begin{gather*}
P_n(t)\equiv \iiint_{\Sigma_+}\sum_{m=0}^n\frac{(\alpha x)^m}{m!}v^2(t,x,y,z)\,dxdydz, \\
Q_n(t)\equiv \iiint_{\Sigma_+}\sum_{m=0}^n\frac{(\alpha x)^m}{m!}|Dv|^2(t,x,y,z)\,dxdydz
\end{gather*}
inequalities
$$
P_n'(t)+\alpha Q_{n-1}(t)+ 2\delta Q_n(t)\leq  \gamma(t) P_n(t)+c, \quad P_n(0)\leq c, \quad 
\|\gamma\|_{L_1(0,T)}\leq c
$$
hold uniformly with respect to $n$,
whence it follows that
\begin{equation}\label{2.9}
\sup_{t\in[0,T]}\iiint_{\Sigma_+}e^{\alpha x} v^2\,dxdydz
+\int_0^T\!\! \iiint_{\Sigma_+}e^{\alpha x} |Dv|^2 \,dxdydzdt<\infty.
\end{equation}
Multiplying the aforementioned equality for the function $v$ by $2e^{\alpha x} v$ and integrating over $\Sigma$ we derive similarly to \eqref{2.8} that for $R(t)\equiv \iiint e^{\alpha x} v^2\,dxdydz$ the following equality holds:
$$
R'(t)-(\alpha^3+\delta\alpha^2)R(t)= g(t)\in L_1(0,T),
$$
therefore, $R \in C[0,T]$. 

Next, multiplying the corresponding equation by
$-2x^m\Delta^{\bot} v$ and integrating over $\Sigma_+$ we derive similarly to \eqref{2.8} that
\begin{multline*}
\frac{d}{dt}\iiint_{\Sigma_+}x^m(v_y^2+v_z^2)\, dxdydz  \\+
m\iiint x^{m-1} (3v_{xy}^2+3v_{xz}^2+(\Delta^{\bot}v)^2-bv_y^2-bv_z^2)\,dxdydz \\ -
m(m-1)(m-2)\iiint_{\Sigma_+}x^{m-3}(v_y^2+v_z^2) \,dxdydz \\+
2\delta \iiint_{\Sigma_+} x^m (v_{xy}^2+v_{xz}^2+(\Delta^{\bot}v)^2)\,dxdydz \\ -
\delta m(m-1)\iiint_{\Sigma_+}x^{m-2}(v_y^2+v_z^2)\,dxdyz =
2\iiint_{\Sigma_+}x^m(\partial^j_xf_y v_y+\partial^j_xf_z v_z)\,dxdydz.
\end{multline*}
Taking into account inequality \eqref{2.5} similarly to \eqref{2.9} one obtains the rest properties for 
$\partial^j_x\partial_{y,z}^\nu u$, $|\nu|\geq 1$. The properties of $\partial_t\partial_x^j u$ are derived with the use of equation \eqref{2.1} itself.
\end{proof}

We now pass to weak solutions. 

\begin{definition}\label{D2.1}
Let $u_0\in L_2^{\psi(x)}$ for an admissible weight function $\psi$, $f\equiv f_0+f_1$, $f\in L_1(0,T; L_2^{\psi(x)})$, 
$f_1\in L_2(0,T;H^{-1,\psi(x)})$. 
A function $u \in L_2(0,T;L_2^{\psi(x)})$ is called a weak solution to problem \eqref{2.1}, \eqref{1.2}, \eqref{1.3}, if for any function $\varphi$, such that 
$\partial^j_x\varphi \in C([0,T]; L_2^{1/\psi(x)})$, 
$\partial_t^m\partial_x^j\partial_{y,z}^\nu \varphi \in L_2(0,T;L_2^{1/\psi(x)})$ for $j\geq 0$, $2m+|\nu|\leq 2$ 
and $\varphi\big|_{t=T}\equiv 0$, $\varphi\big|_{(0,T)\times \partial\Sigma}\equiv 0$, there holds the following equality:
\begin{multline}\label{2.10}
\int_0^T\!\! \iiint \big[u(\varphi_t+b\varphi_x+\Delta\varphi_x+\delta\Delta\varphi) +
f\varphi\big]\,dxdydzdt \\+
\iiint u_0\varphi\big|_{t=0}\,dxdydz=0.
\end{multline}
\end{definition}

\begin{lemma}\label{L2.3}
If there exists $\beta>0$ such that $\psi(x)\geq \varkappa_{0,\beta}(x)$ $\forall x\in\mathbb R$, then
a weak solution to problem \eqref{2.1}, \eqref{1.2}, \eqref{1.3} is unique.
\end{lemma}

\begin{proof} 
The proof is carried out by standard H\"olmgren's argument on the basis of Lemmas~\ref{L2.1} and~\ref{L2.2}. Let $F$ be an arbitrary function from the space $C_0^\infty(\Pi_T)$. Consider an auxiliary linear problem \eqref{2.1}, \eqref{1.2}, \eqref{1.3} for $u_0\equiv 0$ and $f(t,x,y,z)\equiv -F(T-t,-x,y,z)$. According to the aforementioned lemmas there exists a solution $\widetilde\varphi$ to this problem such that 
$\partial_x^j \widetilde\varphi \in C([0,T];L_2^{1/\widetilde\psi(x)})$, where
$\widetilde\psi(x)\equiv \psi(-x)$, $\partial^m_t\partial_x^j \widetilde\varphi, \partial_x^j\partial_{y,z}^\nu \widetilde\varphi\in L_2(0,T;L_2^{1/\widetilde\psi(x)})$ if $2m+|\nu|\leq 2$ (note that $1/\widetilde\psi(x) \leq c e^{2\beta x}$ when $x\to +\infty$, $1/\widetilde\psi(x) \leq c$ when $x\to -\infty$). 

Define $\varphi(t,x,y,z)\equiv \widetilde\varphi(T-t,-x,y,z)$. It is easy to see that this function satisfies the hypothesis of Definition~\ref{D2.1} and $\varphi_t+b\varphi_x+\Delta\varphi_x+\delta\Delta\varphi=F$ in the space
$L_2(0,T;L_2^{1/\psi(x)})$.

Therefore, if $u$ is a weak solution to problem \eqref{2.1}, \eqref{1.2}, \eqref{1.3} for $u_0\equiv 0$ and $f\equiv 0$, it follows from \eqref{2.9} that $\langle u,F\rangle=0$ and so $u\equiv 0$.
\end{proof}

Now we present a number of auxiliary lemmas on solubility of the linear problem in non-smooth case.

\begin{lemma}\label{L2.4}
Let $u_0\in L_2^{\psi(x)}$ for a certain admissible weight function $\psi(x)$ such that $\psi'(x)$ is also an admissible weight function, $f\equiv f_0+\delta^{1/2}f_{1x}+f_{2x}$, where $f_0\in L_1(0,T; L_2^{\psi(x)})$, $f_1\in L_2(0,T;L_2^{\psi(x)})$,
$f_2\in L_2(0,T;L_2^{\psi^2(x)/\psi'(x)})$. Then there exists a weak solution to problem \eqref{2.1}, \eqref{1.2}, \eqref{1.3} $u(t,x,y,z)$ from the space $C([0,T];L_2^{\psi(x)})\cap L_2(0,T;H_0^{1,\psi'(x)})$ and $\delta|Du|\in L_2(0,T;L_2^{\psi(x)})$. Moreover, for any $t\in (0,T]$ uniformly with respect to $\delta$
\begin{multline}\label{2.11}
\|u\|_{C([0,t];L_2^{\psi(x)})}+\|u\|_{L_2(0,t;H^{1,\psi'(x)})}+\delta^{1/2}\bigl\||Du|\bigr\|_{L_2(0,t;L_2^{\psi(x)})} \\ \leq c(T) \left[\|u_0\|_{L_2^{\psi(x)}}+\|f_0\|_{L_1(0,t;L_2^{\psi(x)})}+\|f_1\|_{L_2(0,t;L_2^{\psi(x)})}
+\|f_2\|_{L_2(0,T;L_2^{\psi^2(x)/\psi'(x)})}\right],
\end{multline}
\begin{multline}\label{2.12}
\iiint u^2(t,x,y,z)\psi(x)\,dxdydz+\int_0^t\!\! \iiint (3u_x^2+u_y^2+u_z^2)\psi'\,dxdydzd\tau \\
+2\delta \int_0^t \!\! \iiint |Du|^2\psi\,dxdydzd\tau -
\int_0^t \!\! \iiint u^2\cdot (b\psi'+\psi'''+\delta\psi'')\,dxdydzd\tau \\ =
\iiint u_0^2\psi\,dxdydz+2\int_0^t\!\iiint f_0u\psi \,dxdydzd\tau \\
-2\int_0^t\!\! \iiint (\delta^{1/2}f_1+f_2)(u\psi)_x \,dxdydzd\tau.
\end{multline}
\end{lemma}

\begin{proof}
Let at first $u_0\in C_0^\infty(\Sigma)$, $f_0,f_1,f_2 \in C_0^\infty(\Pi_T)$. Consider the corresponding solution from the class described in Lemmas~\ref{L2.1} and~\ref{L2.2}. Note that $\psi$ is non-decreasing and has not more than exponential growth at $+\infty$. Then 
$\partial_x^j u\in C([0,T];H_0^{1,\psi(x)})$, 
$\partial_x^j u_t, \partial_x^j\partial_{y,z}^\nu u \in L_2(0,T;L_2^{\psi(x)})$ if $|\nu|=2$ for any $j\geq 0$. Therefore, one can multiply equation \eqref{2.1} by $2u(t,x,y,z)\psi(x)$, integrate and as a result obtain equality \eqref{2.12}. Note that this equality provides estimate \eqref{2.11}, which in turn justifies the assertion of the lemma in the general case.
\end{proof}

\begin{corollary}\label{C2.1}
Let the hypothesis of Lemma~\ref{L2.4} be satisfied for $\psi(x)\geq 1$ $\forall x\in\mathbb R$ and $f_2\equiv 0$. Then for the (unique) weak solution $u\in C([0,T];L_2)$ and any $t\in (0,T]$
\begin{multline}\label{2.13}
\iiint u^2(t,x,y,z)\,dxdydz
+2\delta \int_0^t \!\! \iiint |Du|^2\,dxdydzd\tau  =
\iiint u_0^2\,dxdydz \\+2\int_0^t\!\iiint f_0u \,dxdydzd\tau 
-2\delta^{1/2}\int_0^t\! \iiint f_1u_x \,dxdydzd\tau.
\end{multline}
\end{corollary}

\begin{proof}
In the smooth case this equality is obvious and in the general case is obtained on the basis of estimate \eqref{2.11} via closure.
\end{proof}

\begin{lemma}\label{L2.5}
Let $u_0\in H_0^{1,\psi(x)}$ for a certain admissible weight function $\psi(x)$ such that $\psi'(x)$ is also an admissible weight function, $f\equiv f_0+\delta^{1/2}f_1$, where $f_0\in L_1(0,T;H_0^{1,\psi(x)})$, $f_1\in L_2(0,T;L_2^{\psi(x)})$. Then there exists a weak solution to problem \eqref{2.1}, \eqref{1.2}, \eqref{1.3} $u(t,x,y,z)$ from the space $C([0,T];H_0^{1,\psi(x)})\cap L_2(0,T;H^{2,\psi'(x)})$ and $\delta|D^2u|\in L_2(0,T;L_2^{\psi(x)})$ . Moreover, for any $t\in (0,T]$ uniformly with respect to $\delta$
\begin{multline}\label{2.14}
\|u\|_{C([0,t];H^{1,\psi(x)})}+\|u\|_{L_2(0,t;H^{2,\psi'(x)})}+
\delta^{1/2}\bigl\||D^2u|\bigr\|_{L_2(0,t;L_2^{\psi(x)})} \\ \leq c(T) \left[\|u_0\|_{H^{1,\psi(x)}}+
\|f_0\|_{L_1(0,t;H^{1,\psi(x)})}+\|f_1\|_{L_2(0,t;L_2^{\psi(x)})}\right],
\end{multline}
\begin{multline}\label{2.15}
\iiint|Du(t,x,y,z)|^2\psi(x)\,dxdydz
+c_0\int_0^t\!\iiint|D^2u|^2\cdot(\psi'+\delta\psi)\,dxdydzd\tau \\
\leq \iiint|Du_0|^2\psi \,dxdydz+
c\int_0^t\!\iiint |Du|^2\psi\,dxdydzd\tau \\
+2\int_0^t\!\iint(f_{0x}u_x+f_{0y}u_y+f_{0z}u_z)\psi \,dxdydzd\tau \\
-2\delta^{1/2}\int_0^t\iiint f_1[(u_x\psi)_x+u_{yy}\psi+u_{zz}\psi]\,dxdydzd\tau,
\end{multline}
where the constants $c_0$, $c$ depend on $b$ and the properties of the function $\psi$ and the domain $\Omega$.
\end{lemma}

\begin{proof}
In the smooth case $u_0\in C_0^\infty(\Sigma)$, $f_0,f_1 \in C_0^\infty(\Pi_T)$ multiplying \eqref{2.1} by $-2\bigl(u_x(t,x,y,z)\psi(x)\eta_n(x)\bigr)_x -2\Delta^{\bot}u(t,x,y,z)\psi(x)\eta_n(x)$, where $u$ is the solution constructed in Lemmas~\ref{L2.1} and~\ref{L2.2}, $\eta_n(x)\equiv \eta(n-|x|)$, and integrating we obtain an equality
\begin{multline}\label{2.16}
\iiint |Du(t,x,y,z)|^2\psi\eta_n \,dxdydz - \iiint |Du_0|^2\psi\eta_n dxdydz \\+ 
\int_0^t\!\!\iiint(3u_{xx}^2+4u^2_{xy}+4u^2_{xz} +(\Delta^{\bot}u)^2)(\psi\eta_n)' \,dxdydzd\tau \\
+2\delta\int_0^t\!\!\iiint(u^2_{xx}+2u^2_{xy}+2u^2_{xz}+(\Delta^{\bot}u)^2)\psi\eta_n \,dxdydzd\tau \\ 
-\int_0^t\!\!\iiint |Du|^2(b(\psi\eta_n)' +(\psi\eta_n)'''+\delta(\psi\eta_n)'')\,dxdydzd\tau \\
=2\int_0^t\!\!\iiint(f_{0x}u_x+f_{0y}u_y+f_{0z}u_z)\psi\eta_n \,dxdydxd\tau \\
-2\delta^{1/2}\int_0^t\!\!\iiint f_1\bigl((u_x\psi\eta_n)_x+\Delta^{\bot}u\psi\eta_n\bigr)\,dxdydzd\tau. 
\end{multline}
Passing to the limit when $n\to +\infty$ and using the properties of the function $\psi$ and inequality \eqref{2.5} we obtain \eqref{2.15} in the smooth case. This inequality provides estimate \eqref{2.14}. The general case is handled via closure.
\end{proof}

\begin{corollary}\label{C2.2}
Let the hypothesis of Lemma~\ref{L2.5} be satisfied for $\psi(x)\geq 1$ $\forall x\in\mathbb R$. Then for the (unique) weak solution $u\in C([0,T];H_0^1)$ and any $t\in (0,T]$
\begin{multline}\label{2.17}
\iiint |Du(t,x,y,z)|^2\,dxdydz
+2\delta\int_0^t\!\!\iiint(u^2_{xx}+2u^2_{xy}+2u^2_{xz}+(\Delta^{\bot}u)^2)\,dxdydzd\tau  \\ =
\iiint |Du_0|^2\,dxdydz +2\int_0^t\!\iiint(f_{0x}u_x+f_{0y}u_y+f_{0z}u_z) \,dxdydzd\tau \\
-2\delta^{1/2}\int_0^t\!\!\iiint f_1\Delta u\,dxdydzd\tau.
\end{multline}
\end{corollary}

\begin{proof}
In the smooth case this equality is derived from \eqref{2.16}, where formally one must set $\psi\equiv 1$, and the consequent passage to the limit when $n\to +\infty$ and in the general case is obtained on the basis of estimate \eqref{2.14} via closure.
\end{proof}

\begin{lemma}\label{L2.6}
Let the hypothesis of Lemma~\ref{L2.5} be satisfied for some $\delta>0$ and $\psi(x)\geq 1$ $\forall x\in \mathbb R$.  Consider the (unique) weak solution $u\in C([0,T];H_0^{1,\psi(x)})\cap L_2(0,T;H^{2,\psi(x)})$. Then for any $t\in (0,T]$ the following equality holds: 
\begin{multline}\label{2.18}
-\frac 13 \iiint u^3(t,x,y,z)\widetilde\psi(x)\,dxdydz
-b\int_0^t\!\!\iiint u^2u_x\widetilde\psi\,dxdydzd\tau \\
+2\int_0^t\!\!\iiint uu_x \Delta u\widetilde\psi \,dxdydzd\tau 
+\int_0^t\!\!\iiint u^2 \Delta u\widetilde\psi' \,dxdydzd\tau \\
-2\delta\int_0^t\!\!\iiint u |Du|^2\widetilde\psi\,dxdydzd\tau 
-\delta\int_0^t\!\!\iiint u^2u_x\widetilde\psi'\,dxdydzd\tau \\
=-\frac 13 \iiint u_0^3\widetilde\psi \,dxdydz 
-\int_0^t\!\!\iiint f u^2\widetilde\psi \,dxdydzd\tau,
\end{multline}
where either $\widetilde\psi \equiv \psi$ or $\widetilde\psi\equiv 1$.
\end{lemma}

\begin{proof} 
In the smooth case multiplying \eqref{2.1} by $-u^2(t,x,y,z)\widetilde\psi(x)$ and integrating one instantly obtains equality \eqref{2.18}.

In the general case we obtain this equality via closure. Note that by virtue of \eqref{1.8} 
(for $q=4$, $\psi_1=\psi_2=\psi$) if 
$u\in C([0,T];H_0^{1,\psi(x)})\cap L_2(0,T;H^{2,\psi(x)})$ then 
$$
u\in C([0,T];L_4^{\psi(x)}),\qquad |Du|\in L_2(0,T; L_4^{\psi(x)})
$$
and this passage to the limit is easily justified.

\end{proof}

\section{Existence of weak solutions}\label{S3}

Consider the following equation:
\begin{equation}\label{3.1}
u_t+bu_x+\Delta u_x-\delta\Delta u+(g(u))_x=f(t,x,y,z), \quad \delta\in [0,1].
\end{equation}

\begin{definition}\label{D3.1}
Let $u_0\in L_2^{\psi(x)}$ for a certain admissible weight function $\psi(x)\geq 1$ 
$\forall x\in\mathbb R$ such that $\psi'(x)$ is also an admissible weight function, $f\in L_1(0,T; L_2^{\psi(x)})$. 
A function $u \in L_2(0,T;L_2^{\psi(x)})$ is called a weak solution to problem \eqref{3.1}, \eqref{1.2}, \eqref{1.3}, if  for any function $\varphi$, such that $\partial^j_x\varphi \in C([0,T]; L_2^{1/\psi'(x)})$, 
$\partial_t^m\partial_x^j\partial_{y,z}^\nu \varphi \in L_2(0,T;L_2^{1/\psi'(x)})$ for $j\geq 0$, $2m+|\nu|\leq 2$ 
and $\varphi\big|_{t=T}\equiv 0$, $\varphi\big|_{(0,T)\times \partial\Sigma}=0$, the function
$g(u(t,x,y,z))\varphi_x\in L_1(\Pi_T)$ and there holds the following equality:
\begin{multline}\label{3.2}
\int_0^T\!\! \iiint \big[ u(\varphi_t+b\varphi_x+\Delta\varphi_x+\delta\Delta\varphi) +
g(u)\varphi_x + f\varphi\big]\,dxdydzdt \\
+\iiint u_0\varphi\big|_{t=0}\,dxdydz=0.
\end{multline}
\end{definition}

\begin{remark}\label{R3.1}
It is easy to see that if $g\equiv 0$ and a function $u$ is a weak solution to problem \eqref{3.1}, \eqref{1.2}, \eqref{1.3}
in the sense of Definition~\ref{D3.1} then it is a weak solution to this problem in the sense of Definition~\ref{D2.1}.
\end{remark}

\begin{remark}\label{R3.2}
Let $g\in C(\mathbb R)$ and $|g(u)|\leq c(|u|+u^2)$ $\forall u\in \mathbb R$ for a certain constant $c$. Then it easy to see that for any function $u \in L_\infty(0,T;L_2^{\psi(x)})\cap L_2(0,T;H^{1,\psi'(x)})$, where $\psi\geq 1$ is an admissible weight function such that $\psi'$ is also an admissible weight function, and for any function $\varphi$ satisfying the hypothesis of Definition~\ref{D3.1} the function $g(u)\varphi_x\in L_1(\Pi_T)$. In fact, it follows from \eqref{1.8} that
\begin{multline}\label{3.3}
\iiint u^2|\varphi_x|\,dxdydz \leq 
\bigl\|u(\psi')^{1/4}\psi^{1/4}\bigr\|_{L_3}^2 \bigl\|\varphi_x (\psi')^{-1/2}\bigr\|_{L_3} \\ \leq
c\bigl\|(|Du|+|u|)(\psi')^{1/2}\bigr\|_{L_2}\|u\psi^{1/2}\|_{L_2} 
\bigl\|(|D^2\varphi|+|D\varphi|)(1/\psi')^{1/2}\bigr\|_{L_2}.
\end{multline}
\end{remark}

First of all, we prove a lemma on solubility of problem \eqref{3.1}, \eqref{1.2}, \eqref{1.3} for spaces $L_2^{\psi(x)}$ in the  "regularized" case.

\begin{lemma}\label{L3.1}
Let $\delta>0$, $g\in C^1(\mathbb R)$, $g(0)=0$ and $|g'(u)|\leq c\ \forall u\in\mathbb R$. Assume that $u_0\in L_2^{\psi(x)}$ for an admissible weight function $\psi(x)\geq 1$ $\forall x\in\mathbb R$ such that $\psi'(x)$ is also an admissible weight function, $f\in L_1(0,T;L_2^{\psi(x)})$. Then problem \eqref{3.1}, \eqref{1.2}, \eqref{1.3} has a unique weak solution 
$u\in C([0,T];L_2^{\psi(x)})\cap L_2(0,T;H_0^{1,\psi(x)})$.
\end{lemma}

\begin{proof}
We apply the contraction principle. For $t_0\in(0,T]$ define a mapping $\Lambda$ on a set $Y(\Pi_{t_0})=C([0,t_0];L_2^{\psi(x)})\cap L_2(0,t_0;H_0^{1,\psi(x)})$ as follows: $u=\Lambda v\in Y(\Pi_{t_0})$ is a solution to a linear problem
\begin{equation}\label{3.4}
u_t+bu_x+\Delta u_x-\delta\Delta u=f-(g(v))_x
\end{equation}
in $\Pi_{t_0}$ with boundary conditions \eqref{1.2}, \eqref{1.3}.

Note that $|g(v)|\leq c|v|$ and, therefore,
\begin{equation}\label{3.5}
\|g(v)\|_{L_2(0,t_0;L_2^{\psi(x)})}\leq c||v||_{L_2(0,t_0;L_2^{\psi(x)})}<\infty.
\end{equation}

Thus, according to Lemma~\ref{L2.4} (where $f_1\equiv\delta^{-1/2}g(v)$) the mapping $\Lambda$ exists. Moreover, for functions $v,\widetilde{v}\in Y(\Pi_{t_0})$
$$
\|g(v)-g(\widetilde{v})\|_{L_2(0,t_0;L_2^{\psi(x)})}\leq c\|v-\widetilde{v}\|_{L_2(0,t_0;L_2^{\psi(x)})} \leq ct_0^{1/2}\|v-\widetilde{v}\|_{C([0,t_0];L_2^{\psi(x)})}.
$$
As a result, according to inequality \eqref{2.11}
$$
\|\Lambda v-\Lambda\widetilde{v}\|_{Y(\Pi_{t_0})}\leq 
c(T,\delta)t_0^{1/2}\|v-\widetilde{v}\|_{Y(\Pi_{t_0})}.
$$
Since the constant in the right side of this equality is uniform with respect to $u_0$, one can construct the solution on the whole time segment $[0,T]$ by the standard argument.
\end{proof}

Now we pass to the proof of Theorem~\ref{T1.1}.

\begin{proof}[Proof of Theorem~\ref{T1.1}] 
For $h\in (0,1]$ consider a set of initial-boundary value problems in $\Pi_T$
\begin{equation}\label{3.6}
u_t+bu_x+\Delta u_x-h\Delta u+(g_h(u))_x=f(t,x,y,z)
\end{equation}
with boundary conditions \eqref{1.2}, \eqref{1.3}, where
\begin{equation}\label{3.7}
g_h(u)\equiv\int_0^u\Bigl[\theta\eta(2-h|\theta|)+\frac{2\sgn\theta}{h}\eta(h|\theta|-1)\Bigr]\,d\theta.
\end{equation}
Note that $g_h(u)\equiv u^2/2$ if $|u|\leq 1/h$, $|g'_h(u)|\leq 2/h\ \forall u\in\mathbb R$ and $|g'_h(u)|\leq 2|u|$ uniformly with respect to $h$. 

According to Lemma~\ref{L3.1} there exists a unique solution to each of these problems 
$u_h\in C([0,T];L_2^{\psi(x)})\cap L_2(0,T;H_0^{1,\psi(x)})$. Note that similarly to \eqref{3.5} 
$g_h(u_h) \in L_2(0,T;L_2^{\psi(x)})$. 

Next, establish estimates for functions $u_h$ uniform with respect to $h$. 

Write down corresponding equality \eqref{2.13} for functions $u_h$ (we omit the index $h$ in intermediate steps for simplicity):
\begin{multline}\label{3.8}
\iiint u^2 \,dxdydz
+2h\int_0^t\!\iiint |Du|^2\,dxdydzd\tau = \iiint u_0^2 \,dxdydz \\
+2\int_0^t\!\iiint f u \,dxdydzd\tau 
-2\int_0^t\!\iiint g'(u)u_x u \,dxdydzd\tau.
\end{multline}
Since
\begin{equation}\label{3.9}
g'(u)u_xu=\Bigl(\int_0^u g'(\theta)\theta \,d\theta\Bigr)_x \equiv \bigl(g'(u)u\bigr)^*_x,
\end{equation}
where $g^*(u)\equiv \displaystyle \int_0^u g(\theta)\,d\theta$ denotes the primitive for $g$ such that $g^*(0)=0$,
we have that $\iiint g'(u)u_x u\,dxdydz=0$ and equality \eqref{3.8} yields that 
\begin{equation}\label{3.10}
\|u_h\|_{C([0,T];L_2)}+h^{1/2}\|u_h\|_{L_2(0,T;H^1)}\leq c
\end{equation}
uniformly with respect to $h$ (and also uniformly with respect to $\Omega$).

Next, write down corresponding equality \eqref{2.12}, then with the use of \eqref{3.9}
\begin{multline}\label{3.11}
\iiint u^2(t,x,y,z)\psi(x)\,dxdydz+\int_0^t\! \iiint (3u_x^2+u_y^2+u_z^2)\psi'\,dxdydzd\tau \\
+2h \int_0^t \!\! \iiint |Du|^2\psi\,dxdydzd\tau -
\int_0^t \!\! \iiint u^2\cdot (b\psi'+\psi'''+h\psi'')\,dxdydzd\tau \\ =
\iiint u_0^2\psi\,dxdydz +2\int_0^t\!\iiint f u\psi \,dxdydzd\tau \\
+2\int_0^t\! \iiint (g'(u)u)^*\psi' \,dxdydzd\tau.
\end{multline}

Apply interpolating inequality \eqref{1.8} for $k=1$, $m=0$, $q=4$, $\psi_1=\psi_2\equiv\psi'$:
\begin{multline}\label{3.12}
\Bigl|\iiint (g'(u)u )^*\psi'\,dxdydz\Bigr|
\leq \iiint |u|^3\psi'\,dxdydz \\
\leq \Bigl(\iiint u^2\,dxdydz\Bigr)^{1/2} \Bigl(\iiint|u|^4(\psi')^2\,dxdydz\Bigr)^{1/2}
\leq c\Bigl(\iiint u^2\,dxdydz\Bigr)^{1/2} \\ \times \Bigl[\Bigl(\iiint |Du|^2\psi'\,dxdydz\Bigr)^{3/4}
\Bigl(\iiint u^2\psi' \,dxdydz\Bigr)^{1/4}
+\iiint u^2\psi' \,dxdydz\Bigr]
\end{multline}
(note that here the constant $c$ is also uniform with respect to $\Omega$).
Since the norm of the solution in the space $L_2$ is already  estimated in \eqref{3.10}, it follows from \eqref{3.11} and \eqref{3.12} that
\begin{equation}\label{3.13}
\|u_h\|_{C([0,T];L_2^{\psi(x)})}
+\bigl\| |Du_h| \bigr\|_{L_2(0,T;L_2^{\psi'(x)})}
+h^{1/2}\|u_h\|_{L_2(0,T;H^{1,\psi(x)})}\leq c.
\end{equation}

Finally, write down the analogue of \eqref{3.11}, where $\psi(x)$ is substituted by \linebreak $\rho_0(x-x_0)$ for any $x_0\in\mathbb R$. Then it easily follows that (see \eqref{1.6})
\begin{equation}\label{3.14}
\lambda (|Du_h|;T)\leq c.
\end{equation}

In particular, $\|u_h\|_{L_2(0,T;H^1(Q_n))}\leq c(n)$ for any bounded domain $Q_n=(-n,n)\times \Omega$.
Since $|g_h(u)|\leq u^2$ we have that $\|g_h(u_h)\|_{L_\infty(0,T;L_1(Q_n))}\leq c(n)$. Using the well-known embedding $L_1(Q_n)\subset H^{-2}(Q_n)$ we first derive that $\|g_h(u_h)\|_{L_\infty(0,T;H^{-2}(Q_n))}\leq c(n)$, and then according to equation \eqref{3.1} itself that uniformly with respect to $h$
$$
\|u_{ht}\|_{L_1(0,T;H^{-3}(Q_n))}\leq c(n).
$$
Applying the compactness embedding theorem of evolutionary spaces from \cite{S2} we obtain that the set $\{u_h\}$ is precompact in $L_2((0,T)\times Q_n)$ for all $n$.

Now show that if $u_h\to u$ in $L_2((0,T)\times Q_n)$ for some sequence $h\to 0$, then $g_h(u_h)\to u^2/2$ in $L_1((0,T)\times Q_n)$.
Indeed,
\begin{multline*}
|g_h(u_h)-u^2/2|\leq |g_h(u_h)-g_h(u)|+|g_h(u)-u^2/2| \\
\leq 2(|u_h|+|u|)|u_h-u|+|g_h(u)-u^2/2|,
\end{multline*}
where $|g_h(u)-u^2/2|\leq 2u^2\in L_1((0,T)\times Q_n)$ and$g_h(u)\to u^2/2$ pointwise.

As a result, the required solution is constructed in a standard way as the limit of the solutions $u_h$ when $h\to 0$ 
(equality \eqref{3.2} is first derived for the functions $\varphi\eta(n-|x|)$ with consequent passage to the limit when $n\to +\infty$).
\end{proof}

\begin{remark}\label{R3.3}
Theorem~\ref{T1.1} remains valid if $\partial\Omega\in C^2$ but for simplicity we do not present here the corresponding argument.
\end{remark}

We now proceed to solutions in spaces $H_0^{1,\psi(x)}$ and first estimate a lemma analogous to Lemma~\ref{L3.1}.

\begin{lemma}\label{L3.2}
Let $\delta>0$, $g(u)\equiv u^2/2$. Assume that $u_0\in H_0^{1,\psi(x)}$ for an admissible weight function 
$\psi(x)\geq 1$ $\forall x\in\mathbb R$ such that $\psi'(x)$ is also an admissible weight function, 
$f\in L_1(0,T;H_0^{1,\psi(x)})$. Then problem \eqref{3.1}, \eqref{1.2}, \eqref{1.3} has a unique weak solution $u\in C([0,T];H_0^{1,\psi(x)})\cap L_2(0,T;H^{2,\psi(x)})$.
\end{lemma}

\begin{proof}
Introduce for $t_0\in (0,T]$ a space $Y_1(\Pi_{t_0})=C([0,t_0];H_0^{1,\psi(x)})\cap L_2(0,t_0;H^{2,\psi(x)})$ and define a mapping $\Lambda$ on it in the same way as in the proof of Lemma~\ref{L3.1} with the substitution of $Y(\Pi_{t_0})$ by $Y_1(\Pi_{t_0})$ and equation \eqref{3.4} by an equation
$$
u_t+bu_x+\Delta u_x -\delta\Delta u = f - vv_x.
$$
By virtue of \eqref{1.8} (for $\psi_1=\psi_2\equiv \psi\geq 1$)
\begin{multline}\label{3.15}
\|vv_x\|_{L_2(0,t_0;L_2^{\psi(x)})} \leq
\Bigl[\int_0^{t_0}\|v_x\psi^{1/2}\|_{L_4}^2 \|v\psi^{1/2}\|_{L_4}^2\,dt\Bigr]^{1/2} \\ \leq
c\Bigl[\int_0^{t_0} \Bigl(\bigl\||D v_x|\bigr\|_{L_2^{\psi(x)}}^{3/2}\|v_x\|_{L_2^{\psi(x)}}^{1/2}+
\|v_x\|_{L_2^{\psi(x)}}^{2}\Bigr) \bigl\| |Dv|+|v|\bigr\|_{L_2^{\psi(x)}}^{2}\,dt\Bigr]^{1/2} \\ \leq
c_1 t_0^{1/8} \|v\|_{L_2(0,t_0;H^{2,\psi(x)})}^{3/4}\|v\|_{C([0,t_0];H^{1,\psi(x)})}^{5/4} \leq
c_1 t_0^{1/8} \|v\|_{Y_1(\Pi_{t_0})}^2
\end{multline}
and similarly
\begin{equation}\label{3.16}
\|vv_x-\widetilde v\widetilde v_x\|_{L_2(0,t_0;L_2^{\psi(x)})} \leq
c t_0^{1/8} \bigl(\|v\|_{Y_1(\Pi_{t_0})}+\|\widetilde v\|_{Y_1(\Pi_{t_0})}\bigr)
\|v-\widetilde v\|_{Y_1(\Pi_{t_0})}.
\end{equation}
In particular, the hypothesis of Lemma~\ref{L2.5} is satisfied (for $f_1\equiv -vv_x$) and, therefore, the mapping $\Lambda$ exists. Moreover, inequalities \eqref{2.14}, \eqref{3.15}, \eqref{3.16} provide that
\begin{equation}\label{3.17}
\|\Lambda v\|_{Y_1(\Pi_{t_0})}\leq c(T,\delta)\left(\|u_0\|_{H^{1,\psi(x)}}+ 
\|f\|_{L_1(0,T;H^{1,\psi(x)})} +
t_0^{1/8}\|v\|_{Y_1(\Pi_{t_0})}^2\right),
\end{equation}
\begin{equation}\label{3.18}
\|\Lambda v -\Lambda\widetilde v\|_{Y_1(\Pi_{t_0})}\leq c(T,\delta)
t_0^{1/8}\left(\|v\|_{Y_1(\Pi_{t_0})}+\|\widetilde v\|_{Y_1(\Pi_{t_0})}\right)
\|v-\widetilde v\|_{Y_1(\Pi_{t_0})}.
\end{equation}
Existence of unique weak solution to the considered problem in the space $Y_1(\Pi_{t_0})$ on the time interval $[0,t_0]$ depending on $\|u_0\|_{H^{1,\psi(x)}}$ follows from \eqref{3.17}, \eqref{3.18} by the standard argument.

Now we estimate the following a priori estimate: if $u\in Y_1(\Pi_{T'})$ is a solution to the considered problem for some $T'\in (0,T]$ then uniformly with respect to $\delta$
\begin{equation}\label{3.19}
\|u\|_{C([0,T'];H^{1,\psi(x)})}\leq c(T,\|u_0\|_{H^{1,\psi(x)}}, 
\|f\|_{L_1(0,T;H^{1,\psi(x)})}).
\end{equation} 

Note that similarly to \eqref{3.13}
\begin{equation}\label{3.20}
\|u\|_{C([0,T'];L_2^{\psi(x)})} \leq c(T,\|u_0\|_{L_2^{\psi(x)}},\|f\|_{L_1(0,T;L_2^{\psi(x)})}).
\end{equation}

Let either $\widetilde\psi\equiv \psi$ or $\widetilde\psi\equiv 1$
Apply Corollary~\ref{C2.2} for $\widetilde\psi\equiv 1$ or Lemma~\ref{L2.5} for $\widetilde\psi\equiv \psi$, where $f_1\equiv -\delta^{-1/2}uu_x$, then it follows from \eqref{2.17} or \eqref{2.15} that 
\begin{multline}\label{3.21}
\iiint|Du(t,x,y,z)|^2\widetilde\psi(x)\,dxdydz
+c_0\int_0^t\!\iiint|D^2u|^2\cdot(\widetilde\psi'+\delta\widetilde\psi)\,dxdydzd\tau \\
\leq \iiint|Du_0|^2\widetilde\psi \,dxdydz+
c\int_0^t\!\iiint |Du|^2\widetilde\psi\,dxdydzd\tau \\
+2\int_0^t\!\iint(f_{x}u_x+f_{y}u_y+f_{z}u_z)\widetilde\psi \,dxdydzd\tau \\
+2\int_0^t\iiint uu_x[(u_x\widetilde\psi)_x+u_{yy}\widetilde\psi+u_{zz}\widetilde\psi]\,dxdydzd\tau.
\end{multline}
Apply Lemma~\ref{L2.6} then it follows from \eqref{2.18} that
\begin{multline}\label{3.22}
-\frac 13 \iiint u^3(t,x,y,z)\widetilde\psi(x)\,dxdydz
-b\int_0^t\!\!\iiint u^2u_x\widetilde\psi\,dxdydzd\tau \\
+2\int_0^t\!\!\iiint uu_x \Delta u\widetilde\psi \,dxdydzd\tau 
+\int_0^t\!\!\iiint u^2 \Delta u\widetilde\psi' \,dxdydzd\tau \\
-2\delta\int_0^t\!\!\iiint u |Du|^2\widetilde\psi\,dxdydzd\tau 
-\delta\int_0^t\!\!\iiint u^2u_x\widetilde\psi'\,dxdydzd\tau \\
=-\frac 13 \iiint u_0^3\widetilde\psi \,dxdydz 
-\int_0^t\!\!\iiint (f-uu_x) u^2\widetilde\psi \,dxdydzd\tau.
\end{multline}
Summing \eqref{3.21} and \eqref{3.22} provides an inequality
\begin{multline}\label{3.23}
\iiint\Bigl(|Du|^2-\frac {u^3}3\Bigr)\widetilde\psi \,dxdydz
+c _0\int_0^t\!\!\iiint|D^2u|^2\cdot(\widetilde\psi'+\delta\widetilde\psi) \,dxdydzd\tau \\
\leq \iint\Bigl(|Du_0|^2-\frac{u_0^3}3\Bigr)\widetilde\psi \,dxdydz
+c\int_0^t\!\!\iint|Du|^2\widetilde\psi \,dxdydzd\tau \\
+2\int_0^t\!\!\iiint (f_xu_x+f_yu_y+f_zu_z)\widetilde\psi \,dxdydzd\tau 
-\int_0^t\!\!\iiint fu^2\widetilde\psi \,dxdydzd\tau \\ 
+2\int_0^t\!\!\iiint u u^2_x\widetilde\psi' \,dxdydzd\tau
-\int_0^t\!\!\iiint u^2\Delta u\widetilde\psi' \,dxdydzd\tau \\
+2\delta\int_0^t\!\!\iiint u|Du|^2\widetilde\psi \,dxdydzd\tau 
-\frac 13\int_0^t\!\!\iiint u^3\cdot(\delta\widetilde\psi''+b\widetilde\psi') \,dxdydzd\tau \\ 
-\frac 14\int_0^t\!\!\iiint u^4\widetilde\psi' \,dxdydzd\tau.
\end{multline}
By virtue of \eqref{1.8} and \eqref{3.20} similarly to \eqref{3.12}
\begin{multline*}
\iiint |u|^3\widetilde\psi\,dxdydz 
\leq \Bigl(\iiint u^2\,dxdydz\Bigr)^{1/2} \Bigl(\iiint |u|^4 \widetilde\psi^2\,dxdydz\Bigr)^{1/2} \\
\leq c\Bigl[\Bigl(\iiint |Du|^2\widetilde\psi\,dxdydz\Bigr)^{3/4}+1\Bigr].
\end{multline*}
Next,
\begin{multline*}
\iiint |f|u^2\widetilde\psi \,dxdydz\leq c\Bigl(\iiint f^2 \,dxdydz\Bigr)^{1/2} 
\Bigl(\iiint u^4\widetilde\psi^2dxdydz\Bigr)^{1/2} \\
\leq c_1\Bigl(\iint f^2 \,dxdydz\Bigr)^{1/2} 
\Bigl[\Bigl(\iiint |Du|^2\widetilde\psi dxdydz\Bigr)^{3/4}+1\Bigr],
\end{multline*}
\begin{multline*}
\iiint |u|\cdot |Du|^2\widetilde\psi'\,dxdydz \leq \Bigl(\iiint u^2\,dxdydz\Bigr)^{1/2} 
\Bigl(\iiint |Du|^4(\widetilde\psi')^2\,dxdydz\Bigr)^{1/2} \\
\leq c\Bigl[\Bigl(\iiint |D^2 u|^2\widetilde\psi' dxdydz\Bigr)^{3/4}\Bigl(\iiint |Du|^2\widetilde\psi dxdydz\Bigr)^{1/4}+
\iiint |Du|^2\widetilde\psi dxdydz\Bigr],
\end{multline*}
\begin{multline*}
\delta\iiint |u|\cdot|Du|^2\widetilde\psi \,dxdydz \leq \delta\Bigl(\iiint u^2\,dxdydz\Bigr)^{1/2} 
\Bigl(\iiint |Du|^4 \widetilde\psi^2\,dxdydz\Bigr)^{1/2} \\
\leq c\delta\Bigl[\Bigl(\iiint |D^2 u|^2\widetilde\psi dxdydz\Bigr)^{3/4}
\Bigl(\iiint |Du|^2\widetilde\psi dxdydz\Bigr)^{1/4}+
\iiint |Du|^2\widetilde\psi dxdydz\Bigr].
\end{multline*}
Choosing $\widetilde\psi\equiv 1$ we derive from \eqref{3.23} with the use of these estimates that
\begin{equation}\label{3.24}
\|u\|_{C([0,T'];H^1)} \leq c(T,\|u_0\|_{H^1},\|f\|_{L_1(0,T;H^1)}).
\end{equation}
Finally, since
\begin{multline*}
\iiint u^4\psi'\,dxdydz \leq c\Bigl(\iiint u^4\,dxdydz\Bigr)^{1/2}\Bigl(\iiint u^4\psi^2\,dxdydz\Bigr)^{1/2} \\
\leq c_1 \iint \bigl(|Du|^2+u^2\bigr)\,dxdydz \iint \bigl(|Du|^2+u^2\bigr)\psi\,dxdydz, 
\end{multline*}
choosing in \eqref{3.23} $\widetilde\psi\equiv\psi$ with the use of \eqref{3.24} we obtain estimate \eqref{3.19}.

Inequalities \eqref{3.17} and \eqref{3.18} allow us to construct a solution to the considered problem locally in time by the contraction while estimate \eqref{3.19} enables us to extend it for the whole time segment $[0,T]$.
\end{proof}

At the end of this section we present the proof of the part of Theorem~\ref{T1.2} concerning existence of solutions.

\begin{proof}[Proof of Theorem~\ref{T1.2}, existence] 
For $h\in (0,1]$ consider a set of initial-boundary value  problems in $\Pi_T$ 
\begin{equation}\label{3.25}
u_t+bu_x+\Delta u_x -h\Delta u +uu_x =f(t,x,y,z)
\end{equation}
with boundary conditions \eqref{1.2}, \eqref{1.3}. It follows from Lemma~\ref{L3.2} that unique solutions to these problems $u_h\in C([0,T]; H_0^{1,\psi(x)})\cap L_2(0,T;H^{2,\psi(x)})$ exist. Moreover, according to \eqref{3.19} uniformly with respect to $h$
\begin{equation}\label{3.26}
\|u_h\|_{C([0,T];H^{1,\psi(x)})} \leq c.
\end{equation}
Next, inequality \eqref{3.23} in the case $\widetilde\psi\equiv \psi$ applied to the functions $u_h$ provides that uniformly with respect to $h$
\begin{equation}\label{3.27}
\bigl\| |D^2 u_h|\bigr\|_{L_2(0,T;L_2^{\psi'(x)})} \leq c.
\end{equation}
Finally, note that inequality \eqref{3.23} obviously holds for $\widetilde\psi\equiv \rho_0(x-x_0)$ for any 
$x_0\in\mathbb R$, therefore, similarly to \eqref{3.14} 
\begin{equation}\label{3.28}
\lambda(|D^2 u_h|;T)\leq c.
\end{equation}

The end of the proof is exactly the same as for Theorem~\ref{T1.1}.
\end{proof}

\section{Continuous dependence of weak solutions}\label{S4}

Present a theorem from which the result of Theorem~\ref{T1.2} on uniqueness of weak solutions follows.

\begin{theorem}\label{T4.1}
Let $u_0,\widetilde{u}_0\in H_0^{1,\alpha}$, $f,\widetilde{f}\in L_1(0,T;H_0^{1,\alpha})$ for some $\alpha\geq 3/4$, $u,\widetilde{u}$ be weak solutions to corresponding problems \eqref{1.1}--\eqref{1.3} from the class $X^{1,\alpha}(\Pi_T)$. Then for any $\beta>0$
\begin{multline}\label{4.1}
\|u-\widetilde{u}||_{L_\infty(0,T;L_2^{\varkappa_{\alpha,\beta}(x)})}+
\bigl\| |D(u-\widetilde{u})|\bigr\|_{L_2(0,T;L_2^{\varkappa_{\alpha-1/2,\beta}(x)})} \\
\leq c\Bigl(\|(u_0-\widetilde{u}_0)\|_{L_2^{\varkappa_{\alpha,\beta}(x)}} +
\|(f-\widetilde{f})\|_{L_1(0,T;L_2^{\varkappa_{\alpha,\beta}(x)})}\Bigr),
\end{multline}
where the constant $c$ depends on the norms of the functions $u,\widetilde{u}$ in the space $L_\infty(0,T;H^{1,3/4})$.
\end{theorem}

\begin{proof}
Let $\psi(x)\equiv \varkappa_{\alpha,\beta}(x)$, then $\psi'(x)\sim \varkappa_{\alpha-1/2,\beta}(x)$ and
$\psi^2(x)/\psi'(x)\sim \varkappa_{\alpha+1/2,\beta}(x)$. Since $\alpha\geq 1/2$
\begin{multline}\label{4.2}
\iiint u^4\varkappa_{\alpha+1/2,\beta}\,dxdydz 
\leq c\iiint u^4 \rho^2_\alpha\,dxdydz \\ 
\leq c_1\Bigl(\iiint(u^2+|Du|^2)\rho_\alpha\,dxdydz\Bigr)^2,
\end{multline}
we have that $u^2\in L_\infty(0,T;L_2^{\psi^2(x)/\psi'(x)})$. 

Denote $v\equiv u-\tilde{u}$, then the function $v$ is a solution to a linear problem
\begin{equation}\label{4.3}
v_t+bv_x+\Delta v_x=(f-\widetilde{f})-\frac 12 \bigl(u^2-\widetilde{u}^2\bigr)_x\equiv f_0+f_{2x},
\end{equation}
\begin{equation}\label{4.4}
v\big|_{t=0}=u_0-\widetilde{u}_0\equiv v_0,\qquad v\big|_{(0,T)\times\partial\Sigma}=0.
\end{equation}

The hypotheses of Lemmas~\ref{L2.3} and~\ref{L2.4} ($\delta = 0$) hold for this problem, therefore, by virtue of \eqref{2.12}
\begin{multline}\label{4.5}
\iiint v^2\psi\,dxdydz
+\int_0^t\!\!\iiint|Dv|^2\psi'\,dxdydzd\tau \leq \iiint v_0^2\psi\,dxdydz \\
+c\int_0^t\!\!\iiint v^2\psi\,dxdydzd\tau 
+2\int_0^t\!\!\iiint f_0 v\psi\,dxdydzd\tau \\
-2\int_0^t\!\!\iiint f_2(v\psi)_x\,dxdydzd\tau. 
\end{multline}
It is easy to see that
$$
\Bigl|\iiint f_2(v\psi)_x\,dxdydz \Bigr|  
\leq c\iiint \bigl(|u_x|+|\widetilde{u}_x|+|u|+|\widetilde{u}|\bigr)v^2\psi\,dxdydz.
$$
With the use of \eqref{1.8} we derive that
\begin{multline}\label{4.6}
\iiint |u_x|v^2\psi\,dxdydz \\ \leq 
\Bigl(\iiint u^2_x \Bigl(\frac{\psi}{\psi'}\Bigr)^{3/2}\,dxdydz\Bigr)^{1/2}
\Bigl(\iiint v^4(\psi')^{3/2}\psi^{1/2}\,dxdydz\Bigr)^{1/2} \\ 
\leq c\Bigl(\iiint u_x^2\rho_{3/4}\,dxdydz\Bigr)^{1/2} 
\Bigl[\Bigl(\iiint|Dv|^2\psi'\,dxdydz\Bigr)^{3/4}
\Bigl(\iiint v^2\psi\,dxdydz\Bigr)^{1/4} \\ 
+\iiint v^2\psi\,dxdydz\Bigr].
\end{multline}
Other terms are estimated in a similar way and inequality \eqref{4.5} yields the desired result.
\end{proof}

\begin{remark}\label{R4.1}
It easy to see that continuous dependence of solutions can be also established by the same argument in spaces with exponential weights at $+\infty$. More precisely, if $u_0,\widetilde{u}_0\in H_0^{1,\alpha,\exp}$, 
$f,\widetilde{f}\in L_1(0,T;H_0^{1,\alpha,\exp})$ for some $\alpha>0$ then for corresponding weak solutions from the space $X^{1,\alpha,\exp}(\Pi_T)$ the analogue of inequality \eqref{4.1} holds, where the functions $\varkappa_{\alpha,\beta}$, $\varkappa_{\alpha-1/2,\beta}$ are substituted by $e^{2\alpha x}$. Moreover, the constant $c$ in the right side depends here on the norms of functions $u$, $\widetilde u$ in the space $L_\infty(0,T;H^1)$ (since in this case $\psi\sim \psi'$, see \eqref{4.6}). Unfortunately, the applied technique does not allow to avoid exponentially decreasing weight at $-\infty$.
\end{remark}

\section{Large-time decay of solutions}\label{S5}

\begin{proof}[Proof of Theorem~\ref{T1.3}]
Let $\psi(x)\equiv e^{2\alpha x}$ for some $\alpha\in (0,1]$. As in the proof of Theorem~\ref{T1.1} consider the set of solutions $u_h\in C([0,T];L_2^{\alpha,\exp})\cap L_2(0,T;H_0^{1,\alpha\exp})$ to problems \eqref{3.6}, \eqref{1.2}, \eqref{1.3}. Of course, these solutions exist for all positive $T$. 

First of all, note that equality \eqref{3.8} yields here that
\begin{equation}\label{5.1}
\|u_h(t,\cdot,\cdot,\cdot)\|_{L_2} \leq \|u_0\|_{L_2}.
\end{equation}

Now write down equality \eqref{3.11} (for $\psi\equiv e^{2\alpha x}$, we again temporarily omit the index $h$):
\begin{multline}\label{5.2}
\iiint u^2\psi\,dxdydz+2\alpha\int_0^t\!\! \iiint (3u_x^2+u_y^2+u_z^2)\psi\,dxdydzd\tau \\
+2h \int_0^t \!\! \iiint |Du|^2\psi\,dxdydzd\tau-2\alpha(b+4\alpha^2+2h\alpha)\int_0^t \!\! \iiint u^2 \psi\,dxdydzd\tau  \\= \iiint u_0^2\psi\,dxdydz 
+4\int_0^t\!\! \iiint \bigl(g'(u)u\bigr)^*\psi'\,dxdydzd\tau.
\end{multline}
Since $\bigl|\bigl(g'(u)u\bigr)^*\bigr|\leq u^2/h$ it is obvious that $\bigl(g'(u)u\bigr)^*\psi'\in L_\infty(0,T;L_1)$ and from equation \eqref{5.2} follows such an inequality in a differential form: for a.e. $t>0$
\begin{multline}\label{5.3}
\frac d{dt}\iiint u^2\psi\,dxdydz+2\alpha \iiint (3u_x^2+u_y^2+u_z^2)\psi\,dxdydz \\ \leq
2\alpha(b+4\alpha^2+2\alpha) \iiint u^2 \psi\,dxdydz+
2 \iiint \bigl(g'(u)u\bigr)^*\psi'\,dxdydz.
\end{multline}
Continuing inequality \eqref{3.12}, we find with the use of \eqref{5.1} that uniformly with respect to $\Omega$, $h$ and $\alpha$ (see Remark~\ref{R1.3})
\begin{multline}\label{5.4}
\Bigl|\iiint \bigl(g'(u)u\bigr)^*\psi'\,dxdydz\Bigr| \leq 
\alpha \iiint |Du|^2\psi\,dxdydz \\+
c\alpha \bigl(\|u_0\|_{L_2}+\|u_0\|_{L_2}^4\bigr) \iiint u^2\psi\,dxdydz.
\end{multline}
Apply Friedrichs inequality (see, for example, \cite{LSU}): for $\varphi\in H_0^1(\Omega)$
\begin{equation}\label{5.5}
\|\varphi\|_{L_2(\Omega)} \leq c|\Omega|^{1/2}\bigl(\|\varphi_y\|_{L_2(\Omega)}+\|\varphi_z\|_{L_2(\Omega)}\bigr).
\end{equation}
Therefore, for certain constant $c_0$
\begin{equation}\label{5.6}
\iiint (u_y^2+u_z^2)\psi\,dxdydz \geq \frac{c_0}{|\Omega|}\iiint u^2\psi\,dxdydz.
\end{equation}
Combining \eqref{5.3}, \eqref{5.4}, \eqref{5.6} provides that uniformly with respect to $\Omega$, $h$ and $\alpha$
\begin{multline}\label{5.7}
\frac d{dt} \iiint u^2\psi\,dxdydz + \frac{c_0\alpha}{|\Omega|}\iiint u^2\psi\,dxdydz \\ \leq
c\alpha\bigl(b+6\alpha+\|u_0\|_{L_2}+\|u_0\|_{L_2}^4\bigr) \iiint u^2\psi\,dxdydz.
\end{multline}
Choose $\Omega_0 = \displaystyle \frac {c_0}{2cb}$ if $b>0$, $\alpha_0\leq 1$ and $\epsilon_0$ satisfying an inequality
$\displaystyle c(6\alpha_0+\epsilon_0+\epsilon_0^4) \leq \frac {c_0}{4|\Omega|}$, 
$\beta =\displaystyle \frac{c_0}{8|\Omega|}$, then it follows from \eqref{5.7} that uniformly with respect to $h$
\begin{equation}\label{5.8}
\|u_h(t.\cdot,\cdot,\cdot)\|_{L_2^{\psi(x)}} \leq e^{-\alpha\beta t} \|u_0\|_{L_2^{\psi(x)}} \quad \forall t\geq 0.
\end{equation}
Passing to the limit when $h\to+0$ we derive \eqref{1.7}.
\end{proof}

\begin{remark}\label{R5.1}
Besides \eqref{5.5} Friedrichs inequality can be written in another form: if $\Omega\subset (0,L_1)\times (0,L_2)$ then
for $\varphi\in H_0^1(\Omega)$
$$
\|\varphi\|_{L_2(\Omega)} \leq \frac{\min(L_1,L_2)}{\pi}
\bigl(\|\varphi_y\|_{L_2(\Omega)}+\|\varphi_z\|_{L_2(\Omega)}\bigr),
$$
with the corresponding modification of the theorem.
\end{remark}

\section{The initial value problem}\label{S6}

Consider the initial value problem for equation \eqref{1.1} in $\mathbb R^3$ with initial condition \eqref{1.2}. Then all the aforementioned results except Theorem~\ref{T1.3} have analogues for this problem.

In fact, in the smooth case $u_0\in \EuScript S(\mathbb R^3)$, $f\in C^\infty([0,T];\EuScript S(\mathbb R^3))$ a solution to the initial value linear problem \eqref{2.1}, \eqref{1.2} from the space $C^\infty([0,T];\EuScript S(\mathbb R^3))$ can be constructed via Fourier transform. Exponential decay when $x\to+\infty$ similarly to Lemma~\ref{L2.2} can be established for any derivative of this solution. The consequent argument of Sections~\ref{S2}--\ref{S4} can be extended to the case of the initial value problem without  any essential modifications. As a result the following theorems hold.

For a measurable non-negative on $\mathbb R$ function $\psi(x)\not\equiv \text{const}$, let
\begin{equation*}
L_2^{\psi(x)}(\mathbb R^3) =\{\varphi(x,y,z): \varphi\psi^{1/2}(x)\in L_2(\mathbb R^3)\},
\end{equation*}
\begin{equation*}
H^{k,\psi(x)}(\mathbb R^3)=\{\varphi: |D^j\varphi|\in L_2^{\psi(x)}(\mathbb R^3), \  j=0,\dots,k\}
\end{equation*}
endowed with natural norms. Introduce spaces $X^{k,\psi(x)}((0,T\times \mathbb R^3))$, $k=0 \mbox{ or }1$, for admissible non-decreasing weight functions $\psi(x)\geq 1\ \forall x\in\mathbb R$, consisting of functions $u(t,x,y,z)$ such that
$$
u\in C_w([0,T]; H^{k,\psi(x)}(\mathbb R^3)), \qquad
|D^{k+1}u|\in L_2(0,T;L_2^{\psi'(x)}(\mathbb R^3)),
$$
$$
\sup_{x_0\in\mathbb R}\int_0^T\!\! \int_{x_0}^{x_0+1}\!\!\iint_{\mathbb R^2} |D^{k+1}u|^2\,dydzdxdt<\infty
$$
(let $X^{\psi(x)}((0,T)\times \mathbb R^3)=X^{0,\psi(x)}((0,T)\times\mathbb R^3)$). 

\begin{theorem}\label{T6.1}
Let $u_0\in L_2^{\psi(x)}(\mathbb R^3)$, $f\in L_1(0,T; L_2^{\psi(x)}(\mathbb R^3))$ for certain $T>0$ and an admissible weight function $\psi(x)\geq 1\ \forall x\in\mathbb R$ such that $\psi'(x)$ is also an admissible weight function. Then there exists a weak solution to problem \eqref{1.1}, \eqref{1.2} $u \in X^{\psi(x)}((0,T)\times \mathbb R^3))$. 
\end{theorem}

\begin{theorem}\label{T6.2}
Let $u_0\in H^{1,\psi(x)}(\mathbb R^3)$, $f\in L_1(0,T; H^{1,\psi(x)}(\mathbb R^3))$ for certain $T>0$ and an admissible weight function $\psi(x)\geq 1\ \forall x\in\mathbb R$ such that $\psi'(x)$ is also an admissible weight function. Then there exists a weak solution to problem \eqref{1.1}, \eqref{1.2} $u\in X^{1,\psi(x)}((0,T)\times\mathbb R^3))$ and it is unique in this space if $\psi(x)\geq \rho_{3/4}(x)$ $\forall x\in \mathbb R$.  
\end{theorem}

The results on large-time decay can not be established by the same method because of absence of an analogue of Friedrichs inequality in the whole space.

Of course, one can extend the theory to the intermediate cases of domains, for example, 
$\Omega= (0,L)\times \mathbb R$ (here the analogue of Theorem~\ref{T1.3} is also valid, see Remark~\ref{R5.1}).

\end{document}